\documentclass[11pt]{article}
\usepackage[left=1in,right=1in,bottom=1in,top=1in]{geometry}
\usepackage{amsmath}
\usepackage{amsfonts}
\usepackage{amssymb}
\usepackage{amsthm}
\usepackage{young}
\usepackage[vcentermath]{youngtab}
\usepackage{tikz}
\usetikzlibrary{matrix}
\usepackage{enumerate}
\usepackage{enumitem}
\usepackage{tikz-qtree}

\usepackage{algorithm}
\usepackage[noend]{algpseudocode}
\usepackage{amsmath, amssymb, amsthm}
\usepackage{blindtext}
\usepackage[pdftex,bookmarks=true]{hyperref}
\usepackage{cleveref}
\usepackage{comment}
\usepackage{enumerate}
\usepackage{fullpage}
\usepackage{mathtools}
\usepackage{subcaption}

\makeatletter
\newtheorem*{rep@theorem}{\rep@title}
\newcommand{\newreptheorem}[2]{%
\newenvironment{rep#1}[1]{%
 \def\rep@title{#2 \ref{##1}}%
 \begin{rep@theorem}}%
 {\end{rep@theorem}}}
\makeatother

\parskip \medskipamount
\newcommand{\R}{\mathbb R}

\newcommand{\ep}{\varepsilon}

\newtheorem{theorem}{Theorem}[section]
\newtheorem{lemma}[theorem]{Lemma}
\newreptheorem{theorem}{Theorem}

\newtheorem{observation}[theorem]{Observation}

\newcommand{\expect}[1]{\mathbb{E}\!\left[#1\right]}
\newcommand{\E} {\mathbb{E}}

\title{On the spectrum of random walks on complete finite $d$-ary trees}

\author{Evita Nestoridi \thanks{ Department of Mathematics, Princeton University, USA,  emails: exn@princeton.edu, onguyen@princeton.edu. The first author was funded by EP/R022615/1.} \and Oanh Nguyen \footnotemark[1]}

\include{psfig}
\usepackage{mathrsfs}\usetikzlibrary{arrows}
\begin{document}
	\date{}
\maketitle

\begin{abstract}
In the present paper, we determine the full spectrum of the simple random walk on finite, complete $d$-ary trees. We also find an eigenbasis for the transition matrix. As an application, we apply our results to get a lower bound for the interchange process on complete, finite d-ary trees, which we conjecture to be sharp.
\end{abstract}

\section{Introduction}
Finding the spectrum of a transition matrix is a very popular subject in graph theory and Markov chain theory. There are only a few techniques known to describe the exact spectrum of a Markov chain, and they usually work under very specific conditions, such as when the Markov chain is a random walk on a finite group, generated only by a conjugacy class \cite{DiaSha}. Most well-known examples where a transition matrix has been diagonalized usually rely on combination of advanced representation theory, Fourier analysis, and combinatorial arguments \cite{DSaliola},  \cite{Hough}, \cite{Star}, \cite{hyp}, \cite{hyp2}, \cite{hyp3}, \cite{Brown}, \cite{Pike}. But even in most of these cases, there is no description of what an eigenbasis of the transition matrix would look like, which in general is needed as well in order to understand the transition matrix. 

In this work, we present the full spectrum of the simple random walk on complete, finite $d-$ary trees and a corresponding eigenbasis, and we use this information to produce a lower bound for the interchange process on the trees, which we conjecture is sharp. Consider the complete, $d-$ary tree $\mathcal T_h$ of height $h$, which has $n = 1+d+\dots+d^{h} = \frac{d^{h+1}-1}{d-1}$ vertices. We study the simple random walk on $\mathcal T_h$ ,whose transition matrix is denoted by $Q_h$, according to which when we are at the root we stay fixed with probability $1/(d+1)$, or we move to a child with probability $1/(d+1)$ each. When we are at a leaf, we stay fixed with probability $d/(d+1)$ otherwise we move to the unique parent with probability $1/(d+1)$. For any other node, we choose one of the neighbors with probability $1/(d+1)$. 

This is a well studied Markov chain. Aldous \cite{Aldous} proved that the cover time is asymptotic to $2 h^2 d^{h+1} \log h/(h-1)$. The order of the spectral gap and the mixing time of this Markov chain have been widely known for a long time. In fact, the random walk on $\mathcal T_h$ is one of the most famous examples of a random walk not exhibiting cutoff (see Example 18.6 of \cite{AMPY4}). However, finding the exact value of the spectral gap has been an open question for years, let alone finding the entire spectrum and an eigenbasis of the transition matrix $Q_h$.

We denote by $\rho$ the root of $\mathcal T_h$, by $V(\mathcal T_h)$ the vertex set of $\mathcal T_h$, and by $E(\mathcal T_h)$ the set of edges of $\mathcal T_h$. Let $\ell: V(\mathcal T_h) \rightarrow [0,\ldots, h]$ denote the distance from the root. For every node $v$, let $\mathcal T^v$ be the complete $d-$ary subtree rooted at $v$, namely consisting of $v$ and all vertices of $V( \mathcal T_h)$ that are descendants of $v$ in $\mathcal T_h$. Let $\mathcal T_i^v$ be the complete $d-$ary subtree of $\mathcal T^v$ rooted at the $i-$th child of $v$.

The next theorem includes the first result of this paper, presenting the eigenvalues and an eigenbasis of $Q_h$.
\begin{theorem}\label{thm:spectrum} 
	\begin{enumerate}  [label = (\alph*)]
		\item $Q_h$ is diagonalizable with $1$ being an eigenvalue with multiplicity 1. Every other eigenvalue $\lambda\neq 1$ of $Q_h$ is of the form
		\begin{equation}\label{eq:lambda:x:thm}
		\lambda = \frac{d}{d+1}\left (x+ \frac{1}{xd}\right ),
		\end{equation}
		where $x\neq \pm \frac{1}{\sqrt d}$ is a solution of one of the following $h+1$ equations:
		\begin{equation} \label{eq:x:sym:thm}
		d^{h+1}x^{2h+2} = 1
		\end{equation}
		and 
		\begin{equation}  \label{eq:x:antisym}
		d^{k+2} x^{2k+4}- d^{k+2} x^{2k+3}+ dx -1 = 0,\quad \text{for some } 0 \leq k \leq h-1.
		\end{equation}
		
		Reversely, each solution $x\neq \pm \frac{1}{\sqrt d}$ of these equations corresponds to an eigenvalue $\lambda$ according to \eqref{eq:lambda:x:thm}. 
		For each of these equations, if $x$ is a solution then so is $\frac{1}{xd}$. Both $x$ and $\frac{1}{xd}$ correspond to the same $\lambda$. The correspondence between $x$ and $\lambda$ is $2$-to-$1$.
		
		\item \label{thm:spectrum:eigenvector} For each solution $x\neq \pm \frac{1}{\sqrt d}$ of \eqref{eq:x:sym:thm}, an eigenvector $f_{\lambda}$ with respect to $\lambda$ is given by the formula 
		\begin{equation}\label{eq:thm:spectrum:sym}
		f_{\lambda}(v) = \frac{dx^{2}-x}{dx^{2}-1}  x^i +  \frac{x-1}{dx^{2}-1}\frac{1}{d^{i} x^{i}} \quad\text{for every $v$ with $\ell(v)=i$, $0\le i\le h$}.
		\end{equation}
		
		For each $0 \leq k \leq h-1$, each solution $x\neq \pm \frac{1}{\sqrt d}$ of \eqref{eq:x:antisym}, each $v \in V(\mathcal T_h)$ such that $\ell(v)= h-1-k$, and each $j\in [1, \dots, d-1]$, an eigenvector $f_{v, j, j+1}$ with respect to $\lambda$ is given by the formula 
		\begin{equation}\label{eq:thm:spectrum:antisym}
		f_{ v, j, j+1}(w) = 
		\begin{cases}
		& \frac{dx^{i+2}}{dx^{2}-1} - \frac{1}{(dx^{2}-1)d^{i} x^{i}} \quad \mbox{ for } w\in \mathcal T^v_j,\mbox{ where } i= \ell(w) -h+k,\\
		& -\frac{dx^{i+2}}{dx^{2}-1} + \frac{1}{(dx^{2}-1)d^{i }x^{i}} \quad \mbox{ for } w\in \mathcal T^v_{j+1}, \mbox{ where } i= \ell(w) -h+k,\\
		& 0, \mbox{ otherwise.} 
		\end{cases}
		\end{equation}
		\item The collection of these eigenvectors together with the all-1 vector form an eigenbasis of $Q_h$.
	\end{enumerate}

\end{theorem}

In Lemma \ref{lm:sym:antisym} and Figure \ref{fig:sym:antisym}, we describe and illustrate the eigenvectors in more detail.

The idea behind the proof is to consider appropriate projections of the random walk. For example, let $X_t$ be the state of the random walk at time $t$ and let $Y_t$ be the distance of $X_t$ from the root. Then $Y_t$ is a Markov chain on $[0,h],$ whose eigenvalues are also eigenvalues of $Q_h$. Also, the eigenvectors of $Y_t$ lift to give the eigenvectors presented in \eqref{eq:thm:spectrum:sym}. This computation is not going to give us the full spectrum, however. 

For example, in the case of the binary tree, another type of projection to consider is as follows. We consider the process $W_t,$ which is equal to $-Y_t$ if $X_t \in \mathcal T^{\rho}_1$ and equal to $Y_t$ otherwise. The second largest eigenvalue can be derived by this new process, while the eigenvectors are of the form presented in \eqref{eq:thm:spectrum:antisym}. The reason why this is the right process to study is hidden in the mixing time of the random walk on $\mathcal T_h$. A coupling argument roughly says that we have to wait until $X_t$ reaches the root $\rho$. The first time that $X_t$ hits $\rho$ is captured by $W_t$, since $W_t$ is a Markov chain on $[-h,h]$, where the bias is towards the ends and away from zero. The projections that we consider form birth and death processes, whose mixing properties have been thoroughly studied by Ding, Lubetzky, and Peres \cite{DLP}. To capture the entire spectrum, our method is to find in each eigenspace a well-structured eigenvector, which occurs by considering an appropriate projection.

Our analysis has immediate applications to card shuffling, namely the interchange process on $\mathcal T_h$, and to the exclusion process. We enumerate the nodes in $V (\mathcal T _h)$ and we assign cards to the nodes. At time zero, card $i$ is assigned to node $i$. The interchange process on $\mathcal T_h$ chooses an edge uniformly at random and then flips a fair coin. If heads, interchange the cards on the ends of $e$; if tails, stay fixed. A configuration of the deck corresponds to an element of the symmetric group.

Let $g \in S_n$. Let $P$ be the transition matrix of the interchange process on the complete, finite $d-$ary tree $\mathcal T_h$ and let $P^t_{id}(g)$ be the probability that we are at $g$ after $t$ steps, given that we start at the identity. We define the total variation distance between $P^t_{id}$ and the uniform measure $U$ to be 
\begin{equation}
	d(t)= \frac{1}{2} \sum_{x \in S_n} \left \vert P^t_{id}(x) -\frac{1}{n!}\right \vert. \nonumber
\end{equation}

A celebrated result concerning the interchange process was the proof of Aldous conjecture \cite[Theorem 1.1]{CLR}, which states that the spectral gap of P is the same as the spectral gap of the Markov chain that the ace of spades performs. Adjusting our computations, we now get the following result.
\begin{theorem}\label{thm:lowerbound}
	For the interchange process on the complete $d$-ary tree of depth $h$, we have that 
	\begin{itemize}
		\item[(a)]  The spectral gap of the transition matrix is $\frac{(d-1)^{2}}{2(n-1) d^{h+1}} + O \left (\frac{\log_{d} n}{n^{3}}\right )$, 
		\item[(b)] And if $t=\frac{1}{d-1}n^{2}\log n- \frac{1}{d-1}n^2 \log \left( \frac{1}{\ep} \right) + O\left (  n^{2}\right ) $, then 
		$$d(t) \geq1-  \ep,$$
		where $\ep$ is any positive constant. 
	\end{itemize}
\end{theorem}
This is already much faster than the interchange process on the path, another card shuffle that uses $n-1$ transpositions, which Lacoin   \cite{Lacoin} recently proved exhibits cutoff at $\frac{1}{2\pi^2} n^3 \log n$. We conjecture that the lower bound in part $(b)$ of Theorem \ref{thm:lowerbound} is sharp and that the interchange process on $\mathcal T_h$ exhibits cutoff at $\frac{1}{d-1} n^{2}\log n$.

We can get lower bounds for the mixing time of another well studied process, the exclusion process on the complete $d$-ary tree. This is a famous interacting particle system process, according to which at time zero, $k \leq n/2$ nodes of the tree are occupied by indistinguishable particles. At time $t$, we pick an edge uniformly at random and we flip the two ends. Similar computations to the ones of the proof of Theorem \ref{thm:lowerbound} give that if $t=\frac{1}{d-1}n^{2}\log k- \frac{1}{d-1}n^2 \log \left( \frac{1}{\ep} \right) + o\left (  n^{2} \log k\right ) $, then 
$$d(t) \geq1-  \ep,$$
where $\ep>0$ is a constant. Combining Oliveira's result \cite{OL} with Theorem \ref{thm:lowerbound} $(b)$,  we get that the order of the mixing time of the exclusion process on the complete $d-$ary tree is $n^2 \log k$. 

As potential open questions, we suggest trying to find the spectrum or just the exact value of the spectral gap for the simple random on finite Galton-Watson trees or for the frog model as presented in \cite{Jon}.

\section{The spectrum of $Q_h$}
This section is devoted to the proof of Theorem \ref{thm:spectrum}.

Let $\lambda$ be an eigenvalue of $Q_h$ and let $E(\lambda)$ be the corresponding eigenvalue. We first show that there exists an eigenvector in $E(\lambda)$ that has the form described in Theorem \ref{thm:spectrum} \ref{thm:spectrum:eigenvector}.
\begin{lemma}\label{lm:sym:antisym}
The eigenspace $E(\lambda)$ contains an eigenvector $f$ that has one of the following forms:
\begin{enumerate}
\item[(a)] [Completely symmetric] $f(v)=f(w)$ for every $v,w \in V(\mathcal T_h) $ such that $\ell(v)=\ell(w)$. In this case we will call $f$ completely symmetric for $\mathcal T_h$;
\item[(b)] [Pseudo anti-symmetric] There is a node $v$ and $i,j \in \{ 1,\ldots, d\}$ such that  $f(w)=0$ for every $w \notin V(\mathcal T_i^v\cup \mathcal T_j^v)$, $f \vert_{\mathcal T_i^v}$ and $f \vert_{\mathcal T_j^v}$ are completely symmetric, and $f \vert_{\mathcal T_i^v}=-f \vert_{\mathcal T_j^v}$. We call such $f$ pseudo anti-symmetric.
\end{enumerate}
\end{lemma}
The following illustrations explain what the described eigenvectors look like for binary trees.

\tikzset{every tree node/.style={minimum width=2em,draw,circle},
         blank/.style={draw=none},
         edge from parent/.style=
         {draw,edge from parent path={(\tikzparentnode) -- (\tikzchildnode)}},
         level distance=1.5cm}

\begin{figure}[H]
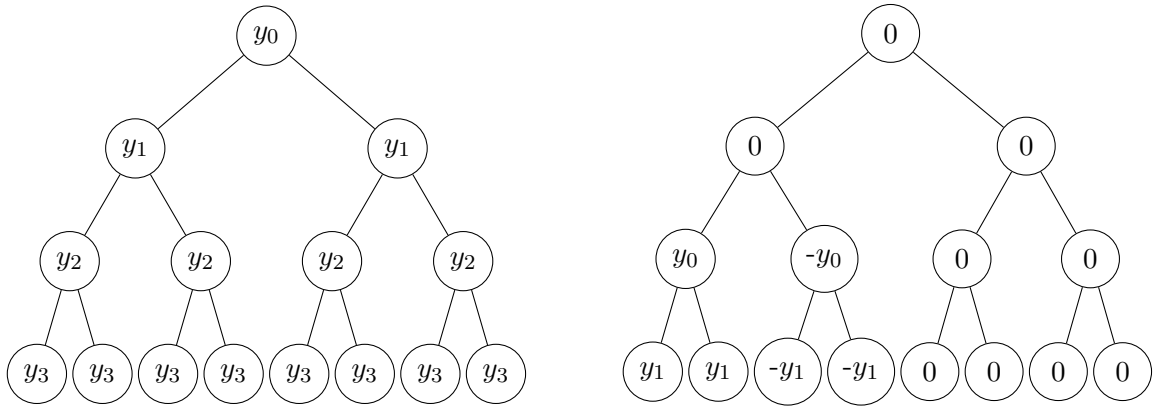
 
	\centering
	\begin{minipage}{.5\textwidth}
		\centering
\Tree
[.$y_0$     
[.$y_1$ 
[.$y_2$
[.$y_3$ ]
[.$y_3$ ]
]
[.$y_2$
[.$y_3$ ]
[.$y_3$ ]
]]
[.$y_1$
[.$y_2$ 
[.$y_3$ ]
[.$y_3$ ]]
[.$y_2$ 
[.$y_3$ ]
[.$y_3$ ]]]
]
	\end{minipage}%
	\begin{minipage}{.5\textwidth}
		\centering
\Tree
[.0     
[.0 
[.$y_0$ 
[.$y_1$ ]
[.$y_1$ ]]
[.-$y_0$ 
[.-$y_1$ ]
[.-$y_1$ ]]]
[.0 
[.0 
[.0 ]
[.0 ]]
[.0 
[.0 ]
[.0 ]
]]
]
	\end{minipage}	
\caption{Completely symmetric eigenvectors (left) and Pseudo anti-symmetric eigenvectors (right)}
\label{fig:sym:antisym}
	\end{figure}

\begin{proof}
Assume that $E(\lambda)$ does not contain a completely symmetric eigenvector. Let $f$ be a nonzero element of $E(\lambda)$. Since $f$ is not completely symmetric, there exist vertices of the same level at which $f$ takes different values. Let $v$ be a vertex with the largest $l(v)$ such that there are at least two of its children, say the $i$-th and $j$-th children, at which $f$ has different values. For example, if there are two leaves $u$ and $w$ at which $f(u)\neq f(w)$ that have the same parent $v'$ then we simply take $v$ to be $v'$. 

By the choice of $v$, $f\vert _{\mathcal T_k^v}$ is completely symmetric for all $k\in [d]$. Indeed, let $u$ be the $k$-th child of  $v$. We have $\mathcal T_k^v = \mathcal T^{u}$. By the choice of $v$, $f$ takes the same value at all children of $u$. Let $u_1, u_2$ be two arbitrary children of $u$. Again by the choice of $v$, $f$ takes the same value, denoted by $f_1$, at all children of $u_1$, and the same value, denoted by $f_2$, at all children of $u_2$. Since $f$ is an eigenvector of $Q_h$, 
\begin{equation}\label{key}
\lambda f(u_1) = \frac{d}{d+1} f_1+\frac{1}{d+1} f(u) \quad \text{and}\quad \lambda f(u_2) = \frac{d}{d+1} f_2+\frac{1}{d+1} f(u). \nonumber
\end{equation}
Since $f(u_1) = f(u_2)$, $f_1 = f_2$. Thus, $f$ takes the same value at all grandchildren of $u$. Repeating this argument shows that  $f\vert _{\mathcal T^u}$ is completely symmetric.

Consider the vector $g$ obtained from $f$ by switching its values on $\mathcal T^{v}_{i}$ and $\mathcal T^{v}_{j}$. More specifically, $g\vert _{\mathcal T^{v}_{i}} = f\vert _{\mathcal T^{v}_{j}}$,  $g\vert _{\mathcal T^{v}_{j}} = f\vert _{\mathcal T^{v}_{i}}$, and $g = f$ elsewhere.
 
 By the symmetry of the tree and the matrix $Q_h$, $g$ also belongs to $E(\lambda)$. So is $f-g$, which we denote by $h$. Observe that $h$ is an eigenvector that is 0 everywhere except on $\mathcal T^{v}_{i} \cup \mathcal T^{v}_{j}$ and $h \vert_{\mathcal T_i^v}=f\vert_{\mathcal T_i^v} - f\vert_{\mathcal T_j^v}=-h \vert_{\mathcal T_j^v}$. Moreover, $h$ is completely symmetric when restricted to $\mathcal T^{v}_{i}$ and $ \mathcal T^{v}_{j}$ because both $f$ and $g$ are, as seen above. Thus, $h \in E(\lambda)$ and is pseudo anti-symmetric.
\end{proof}

\subsection{Completely symmetric eigenvectors}\label{subsection:sym}

In this section, we describe completely symmetric eigenvectors. We shall show that the completely symmetric eigenvectors of $Q_h$ are given by the formula \eqref{eq:thm:spectrum:sym} and correspond to $\lambda$ and $x$ satisfying \eqref{eq:lambda:x:thm} and \eqref{eq:x:sym:thm} as in Theorem \ref{thm:spectrum}.

Since a completely symmetric eigenvector of $Q_h$ has the same value at every node of the same level (see Figure \ref{fig:1}), we can project it onto the path $[0, h]$ and obtain an eigenvector of the corresponding random walk on the path.
 
\begin{figure}[H]\label{figure:sym:3}
	\centering
\begin{tikzpicture}
\Tree
[.$y_0$     
    [.$y_{1}$ 
    [.$y_{2}$ 
     [.$\ldots$ ]
    [.$\ldots$ ]]
    [.$y_{2}$ 
     [.$\ldots$ ]
    [.$\ldots$ ]]]
    [.$y_{1}$ 
      [.$y_2$ 
      [.$\ldots$ ]
      [.$\ldots$ ]]
    [.$y_2$ 
    [.$\ldots$ ]
      [.$\ldots$ ]
    ]]
] 
\end{tikzpicture}
\caption{Completely symmetric eigenvectors}
\label{fig:1}
\end{figure}
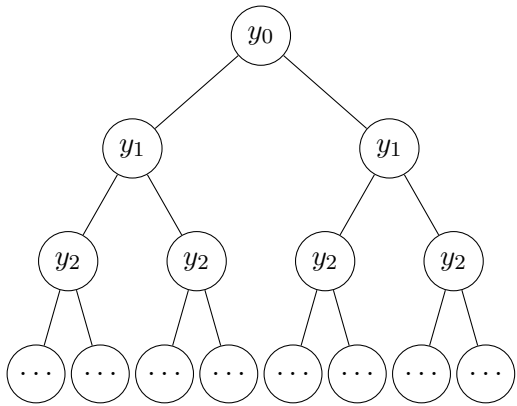

\begin{lemma} \label{lm:sym}
	There are exactly $h+1$ linearly independent completely symmetric eigenvectors of $Q_h$.
\end{lemma}
\begin{proof} Each symmetric eigenvector of $Q_h$ corresponds one-to-one to an eigenvector of the following projection onto the path $[0, h]$ with transition matrix $R_h$:
	\begin{itemize}
		\item $R_h(0, 1) = \frac{d}{d+1},  R_h(0, 0) = \frac{1}{d+1}$,
		\item $R_h(l, l-1) = \frac{1}{d+1}$, $R_h(l, l+1) = \frac{d}{d+1}$ for all $1\le l\le h-1$,
		\item $R_h(h, h-1) = \frac{1}{d+1}$, $R_h(h, h) = \frac{d}{d+1}$,
	\end{itemize}

	Since $R_h$ is a reversible transition matrix with stationary distribution $\pi := [1, d, d^{2}, \dots, d^{h}]$, the matrix $A:= D^{1/2} R_h D^{-1/2}$ is symmetric where $D$ is the diagonal matrix with $D(x, x)= \pi(x)$. Therefore, $A$ is diagonalizable and so is $R_h$. In other words, $R_h$ has $h+1$ independent real eigenvectors. This implies that $Q_h$ has $h+1$ linearly independent completely symmetric eigenvectors.
\end{proof}

\begin{lemma}\label{lm:sym:detail}
	The matrix $R_n$ has 1 as an eigenvalue with multiplicity 1. Each of the remaining $h$ eigenvalues $\lambda\neq 1$ of $R_n$ is of the form 
	$$\lambda = \frac{d}{d+1}\left (x+ \frac{1}{xd}\right )$$
	where $x \neq \pm \frac{1}{\sqrt{d}}$ is a non-real solution of the equation
	\begin{equation} 
	d^{h+1}x^{2h+2} = 1.\nonumber
	\end{equation}
	This equation has exactly $2h$ such solutions. If $x$ is a solution, so is $\frac{1}{xd}$. There is a 2-to-1 correspondence between $x$ and $\lambda$. An eigenvector $y = (y_0, y_1, \dots, y_h)$ of $R_h$ with respect to $\lambda$ is given by  
	\begin{equation} 
	y_i = \frac{dx^{2}-x}{dx^{2}-1}  x^i +  \frac{x-1}{dx^{2}-1}\frac{1}{d^{i} x^{i}} \quad\text{for every $0\le i\le h$}.\nonumber
	\end{equation}
	The vector $f:\mathcal T_h\to \R$ that takes value $y_i$ at all nodes of depth $i$ is an eigenvector of $Q_h$ with respect to $\lambda$.
\end{lemma}
\begin{proof}
	Let $\lambda$ be an eigenvalue of $R_h$ and $y = (y_0, y_1, \dots, y_h)$ be an eigenvector corresponding to $\lambda$. We have
	\begin{enumerate}[label=(R\arabic{*}), ref=R\arabic{*}]
		\item\label{eq:R:0} $\frac{d}{d+1} y_{1} +\frac{1}{d+1} y_{0} = \lambda y_0$,
		\item\label{eq:R:i}  $\frac{1}{d+1} y_{i-1} +\frac{d}{d+1} y_{i+1} = \lambda y_i$ for all $1\le i\le h-1$,
		\item\label{eq:R:h} $\frac{1}{d+1} y_{h-1} +\frac{d}{d+1} y_{h} = \lambda y_h$.
	\end{enumerate}

Since $y$ is not the zero vector, the above equations imply that $y_0\neq 0$. Without loss of generality, we assume $y_0=1$. 
	
	Let $x_1, x_2$ be the solutions to the characteristic equation of \eqref{eq:R:i}: 
	$$\frac{1}{d+1} - \lambda x + \frac{d}{d+1}x^{2} = 0$$
	or equivalently
	\begin{equation}\label{eq:x:lambda:1}
	d x^{2} - (d+1)\lambda x+1 = 0.
	\end{equation}
	
	By \eqref{eq:x:lambda:1}, we have
	$$x_1 x_2 = \frac{1}{d}$$
	and 
	\begin{equation} 
	\lambda = \frac{d}{d+1}(x_1+ x_2) = \frac{d}{d+1}\left (x_1+ \frac{1}{x_1 d}\right ) .\label{eq:lambda:x}
	\end{equation}

	If $x_1\neq x_2$ then we can write $y_0 = \alpha_1 - \alpha_2$, $y_1 = \alpha_1 x_1 -\alpha_2 x_2$ for some $\alpha_1, \alpha_2$. We show that for all $0\le i\le h$,
	\begin{equation}\label{eq:recurrent:y:1}
	y_i = \alpha_1 x_1^{i} - \alpha_2 x_2^{i}.
	\end{equation}
	Indeed, assuming that \eqref{eq:recurrent:y:1} holds for $y_0, \dots, y_i$ for some $1\le i\le h-1$ then by \eqref{eq:x:lambda:1},
	\begin{eqnarray}
	\lambda y_i - \frac{1}{d+1} y_{i-1} = \alpha_1 x_1^{i-1}\left (\lambda x_1 - \frac{1}{d+1}\right )-\alpha_2 x_2^{i-1}\left (\lambda x_2 - \frac{1}{d+1}\right ) = \frac{d}{d+1} \left (\alpha_1 x_1^{i+1}-  \alpha_2 x_2^{i+1}\right ).\nonumber
	\end{eqnarray}
	Thus, by \eqref{eq:R:i}, 
	\begin{equation}\label{key}
	\frac{d}{d+1} y_{i+1} = \frac{d}{d+1} \alpha_1 x_1^{i+1}- \frac{d}{d+1} \alpha_2 x_2^{i+1}\nonumber
	\end{equation}
	and so
	\begin{equation}\label{key}
	y_{i+1} = \alpha_1 x_1^{i+1}- \alpha_2 x_2^{i+1}.\nonumber
	\end{equation}
	Thus, \eqref{eq:recurrent:y:1} also holds for $y_{i+1}$ and hence, for all $y_0, \dots, y_h$. 
	
	Similarly, by \eqref{eq:R:h}, we get
	\begin{eqnarray}
	\frac{d}{d+1} y_h &=&\lambda y_h - \frac{1}{d+1} y_{h-1} = \alpha_1 x_1^{h-1}\left (\lambda x_1 - \frac{1}{d+1}\right )-\alpha_2 x_2^{h-1}\left (\lambda x_2 - \frac{1}{d+1}\right )\nonumber\\
	& =& \frac{d}{d+1} \left (\alpha_1 x_1^{h+1}-  \alpha_2 x_2^{h+1}\right ).\nonumber
	\end{eqnarray}
	Thus, 
	\begin{equation}\label{eq:R:h:1}
	\alpha_1 x_1^{h+1}-  \alpha_2 x_2^{h+1} = \alpha_1 x_1^{h}-  \alpha_2 x_2^{h}
	\end{equation}
	as they are both equal to $y_h$.

	By \eqref{eq:recurrent:y:1}, \eqref{eq:R:0} becomes
	\begin{equation}\label{eq:R:0:1}
	d(\alpha_1x_1 - \alpha_2 x_2) = \left (xd+\frac{1}{x} - 1\right ) (\alpha_1 - \alpha_2).
	\end{equation}

	For simplicity, we write $\alpha = \alpha_1$ and $x = x_1$. By \eqref{eq:recurrent:y:1} for $i=0$, we get
	$$\alpha_2 = \alpha-1.$$

	Equations \eqref{eq:R:0:1} becomes
	\begin{equation} 
	d\alpha x -  \frac{\alpha-1}{x} =  dx+\frac{1}{x} - 1\nonumber
	\end{equation}
	which gives
	\begin{equation}\label{eq:R:0:2}
	\alpha_1 = \alpha = \frac{dx^{2}-x}{dx^{2}-1} \quad\text{and}\quad \alpha_2 = \alpha-1 = \frac{1-x}{dx^{2}-1}. 
	\end{equation}

	Plugging \eqref{eq:R:0:2} into \eqref{eq:R:h:1} and taking into account $x_2 = \frac{1}{xd}$, we get
	\begin{equation}\label{key}
	(dx-1)(x-1)(d^{h+1}x^{2h+2}-1) = 0.\nonumber
	\end{equation}
	
	If $x = 1$ then $\alpha_2 = \alpha-1 = 0$ by \eqref{eq:R:0:2}. And so, $y = \alpha(1, \dots, 1)$ which is an eigenvector of the eigenvalue 1. Since $\lambda\neq 1$, $x\neq 1$. If $x=\frac{1}{d}$ then $x_2 = \frac{1}{xd} = 1$. By the symmetry of $x_1$ and $x_2$, this also corresponds to $\lambda=1$ which is not the case. 
	
	Thus, $x$ satisfies
	\begin{equation} 
	d^{h+1}x^{2h+2}-1=0.\nonumber
	\end{equation}

	This equation has $2h$ non-real solutions and 2 real solutions $\pm \frac{1}{\sqrt{d}}$. For each non-real solution $x_1$, observe that $x_2:=\frac{1}{dx_1}$ is also a non-real solution. Note that $x_1\neq x_2$ and by setting $\lambda$ and $y$ as in \eqref{eq:lambda:x} and \eqref{eq:recurrent:y:1} with $\alpha_1$ and $\alpha_2$ as in \eqref{eq:R:0:2}, one can check that $y$ is indeed an eigenvector corresponding to $\lambda$. Thus, these $2h$ non-real solutions correspond to exactly $h$ eigenvalues $\lambda\neq 1$ of $R_n$. Since $R_n$ has exactly $h+1$ eigenvalues, these are all.
\end{proof}

\subsection{Pseudo anti-symmetric eigenvectors}\label{subsection:anti}
In this section, we describe pseudo anti-symmetric eigenvectors. We shall show that the pseudo anti-symmetric eigenvectors of $Q_h$ are given by the formula \eqref{eq:thm:spectrum:antisym} and correspond to $\lambda$ and $x$ satisfying \eqref{eq:lambda:x:thm} and \eqref{eq:x:antisym} as in Theorem \ref{thm:spectrum}.

Consider a pseudo anti-symmetric eigenvector $f$  with node $v$ and indices $i, j$ as described in Lemma \ref{lm:sym:antisym} (see Figure \ref{fig:sym:antisym}). Let $k = h-\ell(v)-1\in [0, h-1]$. As in Figure \ref{fig:sym:antisym} and Figure \ref{fig:anti}, let $y=(y_0, y_1, \dots, y_k)$ where $y_0$ is the value of $f$ at the $i$-th child of $v$, which is denoted by $u$, $y_1$ is the value of $f$ at the children of $u$ and so on. With these notations, we also write $f$ as $f_{y, v, i, j}$. Observe that $y$ is an eigenvector of the following matrix $S_k$:
\begin{itemize}
	\item $S_k(0, 1) = \frac{d}{d+1}$,
	\item $S_k(l, l-1) = \frac{1}{d+1}$, $S_k(l, l+1) = \frac{d}{d+1}$ for all $1\le l\le k-1$,
	\item $S_k(k, k-1) = \frac{1}{d+1}$, $S_k(k, k) = \frac{d}{d+1}$.
\end{itemize}

Reversely, for any eigenvector $y$ of $S_k$, for any node $v$ at depth $h-k-1$ and for any choice of $i, j\in [1, d]$ with $i\neq j$, we can lift it to a pseudo anti-symmetric eigenvector $f_{y, v, i, j}$.

\begin{figure}[H]	\centering
	\begin{tikzpicture}
	\Tree
	[.0
	[.$0$     
	[.$y_0$ 
	[.$y_1$ 
	[.$y_2$ ]
	[.$y_2$ ]]
	[.$y_1$ 
	[.$y_2$ ]
	[.$y_2$ ]]]
	[.$-y_{0}$ 
	[.$-y_1$ 
	[.$-y_2$ ]
	[.$-y_2$ ]]
	[.$-y_1$ 
	[.$-y_2$ ]
	[.$-y_2$ ]
	]]
	]
	[.$0$     
	[.$0$ 
	[.$0$ 
	[.$0$ ]
	[.$0$ ]]
	[.$0$ 
	[.$0$ ]
	[.$0$ ]]]
	[.$0$ 
	[.$0$ 
	[.$0$ ]
	[.$0$ ]]
	[.$0$ 
	[.$0$ ]
	[.$0$ ]
	]]
	]
	]
	\end{tikzpicture}
	\caption{Pseudo anti-symmetric eigenvectors}\label{fig:anti}

\end{figure}
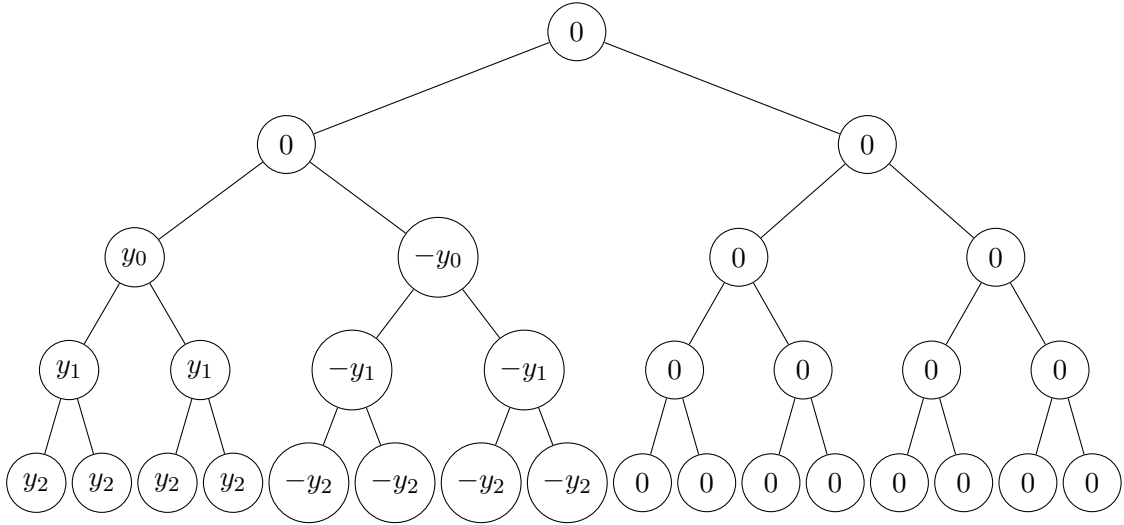

  \begin{lemma} \label{lm:antisym}
	For each $k\in [0, h-1]$, $S_k$ has $k+1$ eigenvectors. For each eigenvector $y$ of $S_k$ and for each $v$ with $l(v) = h-k-1$, there are $d-1$ linearly independent pseudo anti-symmetric eigenvectors of $Q_h$ of the form $f_{y,v,i,j}$.
\end{lemma}
\begin{proof}
	Since $S_k$ differs from $R_k$ only at the $(0, 0)$ entry, it also satisfies the equation $\pi(x) S_k(x, y) = \pi(y)S_k(y, x)$ where $\pi = [1, d, d^{2}, \dots, d^{k}]$. Thus, like $R_k$, the matrix $DS_k D^{-1}$ is symmetric where $D$ is the diagonal matrix with $D(x, x) = \pi(x)^{1/2}$. By symmetry, $D S_k D^{-1}$ has $k+1$ eigenvalues and so does $S_k$. 
	
	For each eigenvector $y$ of $S_k$, we create $d-1$ independent vectors $f_{y, v, i, i+1}$ for $1\le i\le d-1$. It is clear that any $f_{y, v, i, j}$ can be written as a linear combination of these vectors. This completes the proof.
\end{proof}
We now describe the eigenvectors of $S_k$.
\begin{lemma}\label{lm:anti:detail}
	Each of the $k+1$ eigenvalue $\lambda$ of $S_k$ is of the form 
	$$\lambda = \frac{d}{d+1}\left (x+ \frac{1}{dx}\right )$$
	where $x\neq \pm \frac{1}{\sqrt d}$ is a solution of the equation
	\begin{equation} 
	d^{k+2} x^{2k+4}- d^{k+2} x^{2k+3}+ dx -1 = 0.\nonumber
	\end{equation}
	This equation has $2k+2$ solutions that differ from $\frac{1}{\sqrt d}$. If $x$ is a solution, so is $\frac{1}{dx}$. There is a 2-to-1 correspondence between $x$ and $\lambda$. An eigenvector $y = (y_0, y_1, \dots, y_k)$ of $S_k$ with respect to $\lambda$ is given by  
	\begin{equation} 
	y_i = \frac{dx^{i+2}}{dx^{2}-1} - \frac{1}{(dx^{2}-1)d^{i} x^{i}} \quad\text{for every $0\le i\le k$}.\nonumber
	\end{equation}
\end{lemma}

\begin{proof}
	Let $\lambda$ be an eigenvalue of $S_k$ and $y = (y_0, y_1, \dots, y_k)$ be an eigenvector corresponding to $\lambda$. We have
	\begin{enumerate}[label=(S\arabic{*}), ref=S\arabic{*}]
		\item\label{eq:S:0} $\frac{d}{d+1} y_{1} = \lambda y_0$,
		\item\label{eq:S:i}  $\frac{1}{d+1} y_{i-1} +\frac{d}{d+1} y_{i+1} = \lambda y_i$ for all $1\le i\le k-1$,
		\item\label{eq:S:k} $\frac{1}{d+1} y_{k-1} +\frac{d}{d+1} y_{k} = \lambda y_k$.
	\end{enumerate}
	
	As before, we let $x_1, x_2$ be the solutions to the equation 
	$$\frac{1}{d+1} - \lambda x + \frac{d}{d+1}x^{2} = 0.$$
	By exactly the same argument as in the proof of Lemma \ref{lm:sym:detail}, we derive by setting $y_0=1$ that 
	$$y_i = \alpha_1 x_1^{i} - \alpha_2 x_2^{i}$$
where
	$$\alpha_1 = \frac{dx^{2}}{dx^{2}-1}\quad\text{and}\quad \alpha_2 =\frac{1}{dx^{2}-1}$$
	and $x_1$ and $x_2$ satisfy
	\begin{equation}\label{eq:x:2}
	d^{k+2} x^{2k+4}- d^{k+2} x^{2k+3}+ dx -1 = 0
	\end{equation}
	
 	Note that, $x = \pm \frac{1}{\sqrt d}$ are solutions of \eqref{eq:x:2}. The remaining $2k+2$ solutions split into pairs $(x, \frac{1}{dx})$ of distinct components. For each of these pairs, let $x_1 := x$ and $x_2:=\frac{1}{dx}$. We have $x_1\neq x_2$ and by setting $\lambda$ and $y$ as in \eqref{eq:lambda:x} and \eqref{eq:recurrent:y:1} with $\alpha_1 = \frac{dx^{2}}{dx^{2}-1}$ and $\alpha_2 =\frac{1}{dx^{2}-1}$, one can check that $y$ is indeed an eigenvector corresponding to $\lambda$. Thus, these $2k+2$ solutions correspond to exactly $k+1$ eigenvalues $\lambda$ of $S_k$. Since $S_k$ has exactly $k+1$ eigenvalues, these are all.
\end{proof}

\subsection{Proof of Theorem \ref{thm:spectrum}}\label{subsection:proof:spectrum}
The following lemma shows that we can retrieve all eigenvectors of $Q_h$ from completely symmetric and pseudo anti-symmetric eigenvectors. Let $\mathcal A_{S_k}$ be the eigenbasis of $S_k$ as described in Lemma \ref{lm:anti:detail} and $\mathcal B$ be a collection of $h+1$ independent completely symmetric eigenvectors of $Q_h$ as in Lemma \ref{lm:sym:detail}. Let 
$$\mathcal A: = \lbrace f_{y, v, i, i+1}, v \in V(\mathcal T_{h-1}), y \in \mathcal A_{S_{h-\ell(v)-1}}, i \in [d-1] \rbrace.$$
\begin{lemma}\label{lm:spanning}
	The collection $\mathcal A \cup \mathcal B$ is an eigenbasis for $Q_h$.
\end{lemma}

Assuming Lemma \ref{lm:spanning}, we now put everything together to complete the proof of Theorem \ref{thm:spectrum}.
\begin{proof}[Proof of Theorem \ref{thm:spectrum}]
	The first part of the theorem follows from Lemmas \ref{lm:sym:detail} and \ref{lm:anti:detail}. As seen in Lemma \ref{lm:sym:detail}, the set $\mathcal B$ in Lemma \ref{lm:spanning} consists of eigenvectors as in \eqref{eq:thm:spectrum:sym} and the all-1 vector. By Lemmas \ref{lm:antisym} and \ref{lm:anti:detail}, the set $\mathcal A$ consists of eigenvectors as in \eqref{eq:thm:spectrum:antisym}. That gives the second part. Finally, the third part follows from Lemma \ref{lm:spanning}.
\end{proof}

Before proving Lemma \ref{lm:spanning}, we make the following simple observation. For a rooted-tree $T$ that is not necessarily regular, recall that a vector $f: T\to \R$ is said to be \textit{completely symmetric} if $f(u) = f(v)$ for all pairs of vertices $u, v$ at the same level. A vector $f$ is said to be \textit{energy-preserving} if for all level $l$, 
$$\sum_{v\in T: l(v)=l} f(v)=0.$$
\begin{observation}\label{obs}
	For any rooted-tree $T$ and any vector $f: T\to \R$, if $f$ is both energy-preserving and completely symmetric then it is the zero vector.
\end{observation}

 \begin{proof}[Proof of Lemma \ref{lm:spanning}]
First of all, we check that their number is equal to $n$. By Lemmas \ref{lm:sym} and \ref{lm:antisym}, the total number of vectors is
\begin{align*}
h+1+  \sum_{k=0}^{h-1} (k+1)(d-1) d^{h-k-1} 
\end{align*}
where $d^{h-k-1}$ is the number of nodes $v$ of depth $h-k-1$. By algebraic manipulation, this number is exactly $\frac{d^{h+1} -1}{d-1}=n$.

We will now prove that the vectors considered are linearly independent. Assume that there exist coefficients $c_{y,v,i}$ and $c_g$ such that
\begin{equation} 
\sum c_{y,v,i} f_{y,v,i,i+1} + \sum_{g\in \mathcal B} c_g g = 0\nonumber
\end{equation}
where the first sum runs over all $v \in V(\mathcal T_{h-1}), y \in \mathcal A_{S_{h-\ell(v)-1}}, i \in [d-1]$. We need to show that $ c_{y,v,i}$ and $c_g$ are all 0.

Since pseudo anti-symmetric vectors are energy-preserving on $\mathcal T_{h}$, the sum $\sum_{g\in \mathcal B} c_g g = -\sum c_{y,v,i} f_{y,v,i,i+1}$ is both completely symmetric and energy-preserving. And so, by Observation \ref{obs},
\begin{equation}\label{eq:indep:sum}
\sum c_{y,v,i} f_{y,v,i,i+1} = \sum_{g\in \mathcal B} c_g g = 0
\end{equation}
By the independence of vectors in $\mathcal B$, we conclude that $c_g = 0$ for all $g\in \mathcal B$. 
 
We now prove by induction on the vertices of $v\in V(\mathcal T_{h-1})$ and $i\in [d-1]$ that $c_{y,v,i} = 0$ for all $y\in \mathcal A_{S_{h-\ell(v)-1}}$. For this induction, we shall use the natural ordering of pairs $(v, i)$ as follows.
$$(v, i)< (v', i')\quad \text{if and only if} \quad l(v)< l(v') \text{ or } l(v) = l(v') \text{ and } i< i'.$$

 For the base case, which is for $v := \rho$ and $i:=1$, from \eqref{eq:indep:sum}, we have
 \begin{equation}\label{key}
F_{\rho, 1}:= \sum_{y\in \mathcal A_{S_{h-1}}}  c_{y,\rho,1} f_{y,\rho,1,2} = -\sum c_{y,u,j} f_{y,u,j,j+1} \nonumber
 \end{equation}
 where the second sum runs over all $u \in V(\mathcal T_{h-1})$ and $j \in [d-1]$ with $(\rho, 1)< (u, j)$ and all $y \in \mathcal A_{S_{h-\ell(v)-1}}$. Note that when restricting on the subtree $\mathcal T_{1}^{\rho}$, $F_{\rho, 1}$ is a completely symmetric vector because all of the $f_{y,\rho,1,2}$ are completely symmetric. Likewise, $F_{\rho, 1}$ is energy-preserving on $\mathcal T_{1}^{\rho}$, because of the vectors $ f_{y,u,j,j+1}$. By Observation \ref{obs}, $F_{\rho, 1}=0$ on $\mathcal T_{1}^{\rho}$. Since the $f_{y,\rho,1,2}$ are only supported on $\mathcal T_{1}^{\rho}\cup \mathcal T_{2}^{\rho}$ and $f_{y,\rho,1,2}\vert_{\mathcal T_{1}^{\rho}} = -f_{y,\rho,1,2}\vert_{\mathcal T_{2}^{\rho}} $, so is $F_{\rho, 1}$. Therefore, $F_{\rho, 1} = 0$ on $\mathcal T_{2}^{\rho}$ and thus on $\mathcal T_{h}$. So, 
 \begin{equation}\label{key}
\sum_{y\in \mathcal A_{S_{h-1}}}  c_{y,\rho,1} f_{y,\rho,1,2} = 0 \nonumber.
 \end{equation}
 By the independence of vectors in $ \mathcal A_{S_{h-1}}$, we conclude that $c_{y,\rho,1} = 0$ for all $y\in \mathcal A_{S_{h-1}}$, establishing the base case.
 
 For the induction step, assume that for some $(v, i)$, it is proven that $c_{y,w, k} = 0$ for all $(w, k)< (v, i)$ and  $y\in \mathcal A_{S_{h-\ell(w)-1}}$. We now show that $c_{y,v, i} = 0$ for all $y\in \mathcal A_{S_{h-\ell(v)-1}}$. By this assumption, the left-most side in \eqref{eq:indep:sum} reduces to
 \begin{equation} \label{eq:indep:induction}
 \sum c_{y,u, j} f_{y,u, j, j+1} = 0 
 \end{equation}
 where the sum runs over all $(u, j)\ge (v, i)$. Our argument now is similar to the base case. From \eqref{eq:indep:induction}, we have
 \begin{equation} 
 F_{v, i}:=\sum_{y\in \mathcal A_{S_{h-\ell(v)-1}}} c_{y,v,i} f_{y,v,i,i+1} = -\sum_{y, (v, i)<(u, j)} c_{y,u,j} f_{y,u,j,j+1}.\nonumber
 \end{equation}
 Similarly to the base case, when restricting on the subtree $\mathcal T_{i}^{v}$, $F_{v, i}$ is both completely symmetric and energy-preserving on $\mathcal T_{j}^{v}$. By Observation \ref{obs}, $F_{v, i}=0$ on $\mathcal T_{j}^{v}$. This leads to $F_{v, i} = 0$ on $\mathcal T_{i+1}^{v}$ and thus $F_{v, i}=0$ on $\mathcal T_{h}$. So, 
 \begin{equation}\label{key}
 \sum_{y\in \mathcal A_{S_{h-\ell(v)-1}}} c_{y,v,i} f_{y,v,i,i+1}  =0 \nonumber.
 \end{equation}
 By the independence of vectors in $ \mathcal A_{S_{h-\ell(v)-1}}$, we conclude that $c_{y,v, i} = 0$ for all $y\in \mathcal A_{S_{h-\ell(v)-1}}$, establishing the induction step and thus finishing the proof.
\end{proof}

\section{Proof of Theorem \ref{thm:lowerbound}}
\subsection{Proof of Theorem \ref{thm:lowerbound}(a)}
Consider the interchange process on $\mathcal T_h$. Let $Q_h'$ be the transition matrix of the ace of spades. In other words, $Q_h'$ is the transition matrix of any fixed card on the tree. 
By \cite[Theorem 1.1]{CLR}, the spectral gap of the interchange process on the complete $d$-ary tree of depth $h$ is the same as the spectral gap of $Q_h'$. We note that
\begin{equation}\label{key}
Q_h' = \frac{2n-d-3}{2(n-1)} I_n + \frac{d+1}{2(n-1)} Q_h.
\end{equation}
And therefore, the spectral gap of $Q_h'$ is $\frac{d+1}{2(n-1)}$ times the spectral gap of $Q_h$.

Thus \ref{thm:lowerbound} (a) is deduced from the following.
\begin{lemma}\label{lm:Q_h:gap} For sufficiently large $h$, the spectral gap of $Q_h$ is equal $1 - \lambda_2$ where $\lambda_2$ is the second largest eigenvalue of $Q_h$. Moreover,
\begin{equation}\label{eq:lm:gap}
\lambda_2 = 1- \frac{(d-1)^{2}}{(d+1)\cdot d^{h+1}} + O \left (\frac{\log_{d} n}{n^{2}}\right )
\end{equation}
\end{lemma}

To prove Lemma \ref{lm:Q_h:gap}, we shall use Theorem \ref{thm:spectrum}. Let $\lambda$ be an eigenvalue of $Q_h$ and $x$ be a solution of
\begin{equation} \label{eq:x:lambda:1:1}
d x^{2} - (d+1)\lambda x+1 = 0
\end{equation}
which we have encountered in \eqref{eq:x:lambda:1}.

Note that if $\lambda^{2}\ge \frac{4d}{(d+1)^{2}}$ then this equation has two real solutions both of which have the same sign as $\lambda$.
 
Since the equation \eqref{eq:x:sym:thm} only has nonreal solutions except $x = \pm \frac{1}{\sqrt d}$, combining this observation with Theorem \ref{thm:spectrum}, each eigenvalue $\lambda^{2}\ge \frac{4d}{(d+1)^{2}}$ is given by Equation \eqref{eq:lambda:x:thm} for some $x \neq \pm \frac{1}{\sqrt d}$ satisfying
 \begin{equation} \label{eq:x:anti:k}
 d^{k+1} x ^{2k+2}- d^{k+1} x ^{2k+1}+ dx  -1 = 0,
 \end{equation}
for some $k\in [1, h]$ which is simply Equation \eqref{eq:x:antisym} (with $k$ being shifted for notational convenience).

 We shall show the following
 \begin{lemma}\label{lm:bound:lambda}
 	\begin{enumerate}  [label = (\alph*)]
 		\item For all $k\in [1, h]$, Equation \eqref{eq:x:anti:k} has no solutions in $\left (-\infty, -\frac{1}{\sqrt d}\right )$. There are no eigenvalues of $Q_h$ less than $-\sqrt  \frac{4d}{(d+1)^{2}}$.
\item  	 There exists a constant $h_0>0$ such that for all $k\ge h_0$, the largest solution $x$ of \eqref{eq:x:anti:k} satisfies
 	\begin{equation}\label{eq:bound:x}
 	1 - \frac{a}{d^{k+1}} < x<1 - \frac{d-1}{d^{k+1}} \quad\text{where}\quad a=d-1 + \frac{2(d-1)^{2}(k+1)}{d^{k+1}}.
 	\end{equation}
 	Furthermore, for $k=h$, the eigenvalue that corresponds to this $x$ satisfies
 	\begin{equation}\label{eq:bound:lambda}
 	\left |\lambda - \left (1-\frac{(d-1)^{2}}{(d+1)\cdot d^{h+1}}\right )\right | = O\left (\frac{\log_{d} n}{n^{2}}\right ).
 	\end{equation}
 \end{enumerate}
 \end{lemma}
 
 Assuming Lemma \ref{lm:bound:lambda}, we conclude that for sufficiently large $h$, the largest $x$ that satisfies one of the equations \eqref{eq:x:anti:k} for some $k$ in $[1, h]$ satisfies
 \begin{equation} 
 1 - \frac{a}{d^{h+1}} < x<1 - \frac{d-1}{d^{h+1}} \quad\text{where}\quad a=d-1 + \frac{2(d-1)^{2}(h+1)}{d^{h+1}}.\nonumber
 \end{equation}
Since the right-hand side of \eqref{eq:lambda:x:thm} is increasing in $x$ for $x\ge \frac{1}{\sqrt d}$, the second largest eigenvalue $\lambda_2$ of $Q_h$ corresponds to such $x$ and so it satisfies \eqref{eq:bound:lambda}, proving \eqref{eq:lm:gap}. By the first part of Lemma \ref{lm:bound:lambda}, there are no eigenvalues of $Q_h$ whose absolute value is larger than $\lambda_2$. This proves Lemma \ref{lm:Q_h:gap}.
 
 \begin{proof}[Proof of Lemma \ref{lm:bound:lambda}]
 	Let $f(x) = d^{k+1} x^{2k+2} - d^{k+1} x^{2k+1} + dx-1$. 
 	
 	To prove part (a), for all $x< -\frac{1}{\sqrt d}$, we have
 	$$d^{k+1}x^{2k+2}> 1\quad\text{and}\quad - d^{k+1} x^{2k+1} > - dx$$
 	and so $f$ has no roots in $\left (-\infty, -\frac{1}{\sqrt d}\right )$. Assume that there were an eigenvalue $\lambda<-\sqrt  \frac{4d}{(d+1)^{2}}$. By the argument right before \eqref{eq:x:anti:k}, Equation \eqref{eq:x:lambda:1:1} has two negative solutions $x_1<x_2$ with $x_1 x_2=\frac{1}{d}$. We conclude that $x_1<-\frac{1}{\sqrt d}$. This is a contradiction because $x_1$ satisfies \eqref{eq:x:anti:k} for some $k$ while for all $k$, the function $f$ has no roots less than $-\frac{1}{\sqrt d}$.

 	To prove part (b), for all $x\ge \frac{2k+1}{2k+2}$, we have
 	\begin{equation}\label{key}
 	f'(x) = d^{k+1} x^{2k}\left ((2k+2)x-(2k+1)\right ) +d > 0.\nonumber
 	\end{equation}
 	Thus, $f$ is increasing on the interval $[1 - \frac{1}{2k+2} , \infty)$ which contains $[1 - \frac{a}{d^{k+1}}, 1 - \frac{d-1}{d^{k+1}}]$ for sufficiently large $k$. Thus, to prove \eqref{eq:bound:x}, it suffices to show that 
 	\begin{equation}\label{eq:derivative:test}
 	f\left (1 - \frac{a}{d^{k+1}}\right )<0<f\left (1 - \frac{d-1}{d^{k+1}}\right )
 	\end{equation}
 	for sufficiently large $k$. 
 	Indeed, 
 	\begin{equation}\label{key}
 	f\left (1 - \frac{a}{d^{k+1}}\right ) = d-1 - a\left (1 - \frac{a}{d^{k+1}}\right )^{2k+1} -\frac{a}{d^{k}} < d-1- a \left (1 - \frac{a}{d^{k+1}}\right )^{2k+1}<d-1- a \left (1 - \frac{a(2k+1)}{d^{k+1}}\right ) \nonumber
 	\end{equation}
 	which is, by plugging in $a=d-1 +\frac{2(d-1)^{2}(k+1)}{d^{k+1}}$, at most
 	\begin{equation}\label{key}
 	-\frac{2(d-1)^{2}(k+1)}{d^{k+1}} + \frac{a^{2}(2k+1)}{d^{k+1}}\le 0.\nonumber
 	\end{equation}
 	Thus the first inequality of \eqref{eq:derivative:test} holds. 
 	For the second inequality, we have
 	\begin{eqnarray} 
 	f\left (1 - \frac{d-1}{d^{k+1}}\right ) &=& d-1 - (d-1)\left (1 - \frac{d-1}{d^{k+1}}\right )^{2k+1} -\frac{d-1}{d^{k}}\nonumber\\
 	& \ge& d-1 - (d-1)\left (1 - \frac{3(d-1)}{d^{k+1}}\right )  -\frac{d-1}{d^{k}} >0,\nonumber
 	\end{eqnarray}
 	proving the \eqref{eq:derivative:test}. 
 	
 	We have shown that there exists a solution $x = 1-\alpha$ where $\frac{d-1}{d^{k+1}}\le \alpha \le \frac{a}{d^{k+1}}$. Let $\lambda$ be the eigenvalue corresponding to $x$ as in \eqref{eq:lambda:x}. We have
 	\begin{equation}\label{key}
 	\frac{d+1}{d}\lambda = 1 - \alpha+\frac{1}{d(1-\alpha)} \in  \left (1 - \alpha+\frac{1}{d} (1 +\alpha), 1 - \alpha+\frac{1}{d} (1 +\alpha+2\alpha^{2})\right ). \nonumber 
 	\end{equation}
 	In other words, 
 	\begin{equation}\label{key}
 	\frac{d+1}{d}\lambda \in\left (\frac{d+1}{d} -\frac{d-1}{d} \alpha , \frac{d+1}{d} -\frac{d-1}{d} \alpha +\frac{2}{d}\alpha^{2}\right ). \nonumber
 	\end{equation}
 	Using the bounds $\frac{d-1}{d^{k+1}}\le \alpha \le \frac{a}{d^{k+1}}$, we obtain
 	\begin{equation}\label{key}
 	\lambda - \left (1-\frac{(d-1)^{2}}{(d+1)\cdot d^{k+1}}\right )  \le  \frac{2}{d+1}\alpha^{2}\le \frac{2a^{2}}{(d+1)\cdot d^{2k+2}}\le \frac{2}{(d+1)\cdot d^{2k+1}}\nonumber 
 	\end{equation}
 	and
 	\begin{equation}\label{key}
 	\lambda - \left (1-\frac{(d-1)^{2}}{(d+1)\cdot d^{k+1}}\right )  \ge  -\frac{d-1}{d+1}\alpha+\frac{(d-1)^{2}}{(d+1)\cdot d^{k+1}} \ge -\frac{2(d-1)^{3}(k+1)}{(d+1)\cdot d^{2k+2}}\ge - \frac{2(k+1)}{d^{2k}}.\nonumber 
 	\end{equation}
 	Thus, for $k=h$,
 	\begin{equation} 
 	\left |\lambda - \left (1-\frac{(d-1)^{2}}{(d+1)\cdot d^{h+1}}\right )\right | \le \frac{2(h+1)}{d^{2h}} .\nonumber
 	\end{equation}
 	These bounds together with the equation $n = \frac{d^{h+1}-1}{d-1} \in (d^{h}, 2d^{h})$ give \eqref{eq:bound:lambda}.
 \end{proof}

 \subsection{Proof of Theorem \ref{thm:lowerbound} (b)}
 For the proof of the lower bound, we will use Wilson's lemma.
 \begin{lemma}[Lemma 5, \cite{Wilson}]\label{W}
 	Let $\ep, R$ be positive numbers and $0<\gamma< 2-\sqrt{2} $. Let $F: X\to \R$ be a function on the state space $X$ of a Markov chain $(C_t)$ such that 
 	$$\expect{F(C_{t+1})\vert C_t) }= (1 - \gamma )F(C_t), \quad \expect{\left [F(C_{t+1})- F(C_{t})\right ]^2 \vert C_t} \leq  R,$$ and 
 	$$t \leq \frac{ \log \max_{x\in X}F(x) + \frac{1}{2} \log( \gamma  \varepsilon/(4R))}{-\log (1 - \gamma )}.
 	$$ Then the total variation distance from stationarity at time $t$ is at least $1-\ep$.
 \end{lemma}

 \begin{proof}[Proof of Theorem \ref{thm:lowerbound} (b)]
 Let $0<x<1$ be a solution of \eqref{eq:x:antisym} (for $k=h$) satisfying \eqref{eq:bound:x} and $\lambda$ be the eigenvalue of $Q_h$ corresponding to $x$. In particular,
\begin{equation}
\lambda = \frac{d}{d+1}\left (x+ \frac{1}{dx}\right ).\nonumber
\end{equation}
Let $f:\mathcal T_{h}\to \R$ be an eigenvector of $Q_h$ corresponding to $\lambda$. As in the proof of Lemma \ref{lm:anti:detail}, we can choose $f$ as follows.
\begin{equation}\label{eq:def:f}
f(v) = \begin{cases}
0\quad \text{if } v\notin T_{1}^{\rho}\cup T_{2}^{\rho},  \\
dx^{l(v)+2} - \frac{1}{d^{l(v)-1}x^{l(v)-2}}, \quad \text{if } v\in T_{1}^{\rho},\\
-dx^{l(v)+2} + \frac{1}{d^{l(v)-1}x^{l(v)-2}}, \quad \text{if } v\in T_{2}^{\rho}.
\end{cases}
\end{equation}

We now consider the interchange process on the $d$-ary tree $\mathcal T_h$. Fix an arbitrary enumeration of the vertices of $\mathcal T_{h}$ by $1, 2, \dots, n$. Let $\sigma \in S_n$. We define $F (\sigma)= \sum_{v=1}^n f(v) f(\sigma(v))$.
Then, we have that 
\begin{eqnarray}
\expect{F(\sigma_{t+1})\vert \sigma_t) }= \frac{1}{n-1}\sum_{e} \E\left (F(\sigma_{t+1})\vert \sigma_t, e\right ), \nonumber
\end{eqnarray}
where the sum runs over all $n-1$ edges $e$ of the tree and the conditioning on the right is conditioning on the edge $e$ being chosen. 
So, 
\begin{eqnarray}
\expect{F(\sigma_{t+1})\vert \sigma_t) }= \frac{1}{n-1}\sum_{e} \sum_{v=1}^{n} f(v)  \E\left [f({\sigma}_{t+1}(v))\vert \sigma_t, e\right ]=\frac{1}{n-1}  \sum_{v=1}^{n} f(v) \sum_{e} \E\left [f({\sigma}_{t+1}(v))\vert \sigma_t, e\right ]\nonumber.
\end{eqnarray}

By direct computation, we obtain
\begin{equation}\label{key}
\sum_{e} \E\left (f(\sigma_{t+1}(v)\vert \sigma_t, e\right )= \left (n-\frac{d}{2}-\frac{3}{2} + \frac{\lambda(d+1)}{2}\right ) f(\sigma_t(v)).\nonumber
\end{equation}

So, we have
\begin{eqnarray}
\expect{F(\sigma_{t+1})\vert \sigma_t) }= \frac{n-\frac{d}{2}-\frac{3}{2} + \frac{\lambda(d+1)}{2}}{n-1} F(\sigma_t)\nonumber.
\end{eqnarray}

For $n $ sufficiently big, we  have that 
$$F(id)= \sum_{v=1}^n f^{2}(v)  \geq 2 d^{h-1} \left (dx^{h+2} - \frac{1}{d^{h-1}x^{h-2}}\right )^2 \geq \frac{dn}{2},$$
where the first inequality occurs from keeping only the leaves of $T_{1}^{\rho}\cup T_{2}^{\rho}$ and the last inequality follows from \eqref{eq:bound:x}.

Finally, we consider what happens if we change a configuration $\sigma_t$ by transposing an edge $e$, which connects two vertices $u$ and $v$. We have that 
$$| F(\sigma_{t+1})- F(\sigma_t)|= |(f(u)-f(v)) (f(\sigma(u)) - f(\sigma(v)))|.$$
Note that for all vertices $w$, we have by definition of $f$,
$$|f(w)| \le d$$
and by \eqref{eq:bound:x}, assuming wlog that $l(u) = l(v)+1 =: l+1$,
$$|f(v) - f(u)|\le d x^{l+2} (1-x) + \frac{1}{d^{l-1}x^{l-2}}\left (1 - \frac{1}{dx}\right )\le \frac{1}{d^{h-1}} + \frac{1}{d^{l-1}}\le \frac{2}{d^{l-1}}.$$
Thus, $|(f(u)-f(v)) (f(\sigma(u)) - f(\sigma(v)))|^{2}\le \frac{16}{d^{2l-4}}$. Note that by \eqref{eq:def:f}, the left-hand side is 0 if neither $u$ nor $v$ belongs to $T_{1}^{\rho}\cup T_{2}^{\rho}$. And so,
$$\E_{e} \left (\left (F(\sigma_{t+1})- F(\sigma_t)\right )^{2}\vert \sigma_t\right ) \le \frac{1}{n-1} \sum_{l=0}^{h} 2d^{l} \frac{16}{d^{2l-4}} \le \frac{64 d^{4}}{n-1},$$
where $2d^{l}$ is the number of edges $e$ that connect levels $l$ and $l+1$ of $T^{1}_{\rho}\cup T^{2}_{\rho}$.

So we can take $R = \frac{64d^{4}}{n-1}$.

Set $\lambda'= \frac{n-\frac{d}{2}-\frac{3}{2} + \frac{\lambda(d+1)}{2}}{n-1}$. Using Wilson's lemma, we have that if $t \le t_0:=\frac{ \log (F(id)) + \frac{1}{2} \log( (1-\lambda') \varepsilon/(4R))}{-\log (\lambda')}$ then the total variation distance is at least $1- \ep$. By Lemma \ref{lm:Q_h:gap}, 
$$\lambda=  1-\frac{(d-1)^{2}}{(d+1)\cdot d^{h+1}}  + O\left (\frac{\log_{d} n}{n^{2}}\right ),$$ 
which gives
$$1 - \lambda' = \frac{(d+1)(1-\lambda)}{2(n-1)} =\frac{(d-1)^{2}}{2(n-1)d^{h+1}} + O\left (\frac{d\log_{d} n}{n^{3}}\right ).$$
we get
\begin{eqnarray}
t_0 &=& \frac{\frac{\log n}{2}+\frac{\log \ep}{2}+O(\log d)}{\frac{(d-1)^{2}}{2(n-1)d^{h+1}} + O\left (\frac{d\log_{d} n}{n^{3}}\right )} =\left (\frac{\log n}{2}+\frac{\log \ep}{2}+O(\log d)\right )\left (\frac{2n^{2}}{d-1} + O\left (n\log_{d}n\right )\right )  \nonumber\\
&=& \frac{n^{2}\left (\log n+\log \ep\right ) }{d-1}+ O\left (n^{2}\right )\nonumber.
\end{eqnarray}
This completes the proof.
\end{proof}

\bibliographystyle{plain}
\bibliography{tree2}
\end{document}